\documentclass[11pt,oneside]{amsart}
\usepackage{amsxtra}
\usepackage{color}
\usepackage{amsopn}
\usepackage{amsmath,amsthm,amssymb}
\usepackage{mathrsfs,mathtools}
\usepackage{mathrsfs}
\usepackage{hyperref}
\usepackage{enumerate}
\usepackage{xcolor}

\theoremstyle{theorem}
\newtheorem{theorem}{Theorem}[section]
\newtheorem{corollary}[theorem]{Corollary}
\newtheorem{proposition}[theorem]{Proposition}
\newtheorem{lemma}[theorem]{Lemma}
\newtheorem*{theorem*}{Theorem}
\theoremstyle{definition}
\newtheorem{definition}[theorem]{Definition}

\newtheorem{remark}[theorem]{Remark}

\newcommand{\R}{{\mathbb R}}

\newcommand{\C}{{\mathbb C}}

\newcommand{\f}{\varphi}
\newcommand{\Ad}{{\rm Ad}}
\newcommand{\U}{{\rm U}}

\newcommand{\A}{{\rm A}}
\newcommand{\G}{{\rm G}}
\newcommand{\K}{{\rm K}}
\newcommand{\Fg}{{\rm F}}

\newcommand{\Hg}{{\rm H}}

\newcommand{\dt}{\frac{d}{dt}}
\newcommand{\dtz}{\left.\dt\right|_{t=0}}

\newcommand{\fra}{\mathfrak{a}}
\newcommand{\franp}{\mathfrak{a_\nu^\perp}}
\newcommand{\frg}{\mathfrak{g}}

\newcommand{\frh}{\mathfrak{h}}
\newcommand{\frk}{\mathfrak{k}}

\newcommand{\frp}{\mathfrak{p}}

\newcommand{\fru}{\mathfrak{u}}

\newcommand{\MM}{\mathscr{M}(M)}
\newcommand{\CM}{\mathcal{C}(M)}
\newcommand{\PM}{\mathscr{P}(M)}
\newcommand{\PR}{\mathscr{P}}
\newcommand{\F}{\mathfrak{F}}
\newcommand{\Fa}{\mathfrak{F}_\fra}
\newcommand{\Os}{\mathcal{O}}

\newcommand{\st}{\ |\ }

\newcommand{\ts}{\overline{t}}
\newcommand{\Aff}{{\rm Af{}f}}

\numberwithin{equation}{section}

\title{Convexity theorems for the gradient map on probability measures}
\author{Leonardo Biliotti and Alberto Raffero}
\subjclass[2010]{53D20}
\keywords{gradient map, probability measures, convexity}
\thanks{The authors were partially supported by FIRB  ``Geometria differenziale e teoria geometrica delle funzioni'' of MIUR, and by GNSAGA of INdAM.
The first author was also supported by PRIN ``Variet\`a reali e complesse: geometria, topologia e analisi armonica'' of MIUR}
\address{(Leonardo Biliotti) Dipartimento di Matematica e Informatica, Universit\`a degli Studi di Parma, Parco Area delle Scienze 53/A, 43124 Parma, Italy}
\email{leonardo.biliotti@unipr.it}
\address{(Alberto Raffero) Dipartimento di Matematica e Informatica ``U.~Dini'' \\ Universit\`a degli Studi di Firenze\\ Viale Morgagni 67/a\\ 50134 Firenze\\ Italy}
\email{alberto.raffero@unifi.it}

\begin{document}
\maketitle
\begin{abstract}
Given a K\"ahler manifold $(Z,J,\omega)$ and a compact real submanifold $M\subset Z$,
we study the properties of the gradient map
associated with the action of a noncompact real reductive Lie group $\G$
on the space of probability measures on $M.$ In particular, we prove convexity results for such map when $\G$ is Abelian
and we investigate how to extend them to the non-Abelian case.
\end{abstract}

\section{Introduction}
Let $(Z,J,\omega)$ be a compact connected K\"ahler manifold and let $\U$ be a compact connected Lie group with Lie algebra $\fru$.
Assume that $\U$ acts on $Z$ by holomorphic isometries and in a Hamiltonian fashion with momentum mapping $\mu:Z\rightarrow\fru^*$.
It is well-known that the $\U$-action extends to a holomorphic action of the complexification $\U^\C$ of $\U$.
Moreover, the latter gives rise to a continuous action of $\U^\C$ on the space of Borel probability measures on $Z$ endowed with the weak* topology.
We denote such space by $\mathscr{P}(Z)$.

Recently, the first author and Ghigi \cite{BilGhi} studied the properties of the $\U^\C$-action on $\mathscr{P}(Z)$ using momentum mapping techniques.
Although it is still not clear whether any reasonable symplectic structure on $\mathscr{P}(Z)$ may exist (but see \cite{GaKiPa} for something similar on the Euclidean space),
in this setting it is possible to define an analogue of the momentum mapping, namely
\[
\F : \mathscr{P}(Z) \rightarrow  \fru^*, \quad  \F (\nu)=\int_Z \mu(z) \mathrm{d} \nu (z).
\]
$\F$ is called {\em gradient map}. Using it, the usual concepts of stability appearing in K\"ahler geometry
\cite{GeoRobSal,Hei,HeiHuc,HeiHucLoo,HeiLoo,KemNes,Kir,MumFogKir,Schw,Sja}
can be defined for probability measures, too.

In \cite{BilGhi}, the authors were interested in determining the conditions for which the $\U^\C$-orbit of a given probability measure $\nu\in\mathscr{P}(Z)$ has
non-empty intersection with $\F^{-1}(0)$, whenever $0$ belongs to the convex hull of $\mu(Z)$.
This problem is motivated by an application to upper bounds for the first eigenvalue of the Laplacian acting on functions (see also \cite{ArGhLo, BilGhi2, BilGhi0, BouLiYau, Her}).
Furthermore, they obtained various stability criteria for measures.

Stability theory for the action of a compatible subgroup $\G$ of $\U^\C$ was analyzed by the first author and Zedda in \cite{BilZed}.

Recall that a closed subgroup $\G$ of $\U^\C$ is called {\em compatible} if the Cartan decomposition $\U^\C =\U\exp (i \fru)$
induces a Cartan decomposition $\G=\K\exp (\frp)$, where $\K\coloneqq\G\cap \U$ and $\frp\coloneqq \frg\cap i \fru$ is a $\K$-stable linear subspace of $i\fru$.

Identify $\fru^*$ with $\fru$ by means of an $\Ad(\U)$-invariant scalar product on $\fru$.
For each $z\in Z$, let $\mu_\frp(z)$ denote $-i$ times the component of $\mu(z)$ in the direction of $i\frp\subset\fru$. This defines a $\K$-equivariant map
$\mu_{\frp} :Z \longrightarrow \frp$, which is called $\G$-\emph{gradient map} associated with $\mu$ \cite{HeiSchu,HeiSchSto,HeiSto}.
Since $\U^\C$ acts holomorphically on $Z,$ the fundamental vector field $\beta_Z\in\mathfrak{X}(Z)$ of any $\beta\in\frp$ is the gradient of the function
$\mu_\frp^\beta(\cdot) := \langle\mu_\frp(\cdot),\beta\rangle$ with respect to the Riemannian metric $\omega(\cdot,J\cdot)$,
$\langle\cdot,\cdot\rangle$ being an $\Ad(\K)$-invariant scalar product on $\frp$.

If $M$ is a compact $\G$-stable real submanifold of $Z$, we can restrict $\mu_{\frp}$ to $M.$
Moreover, the $\G$-action on $M$ extends in a natural way to a continuous action on $\PM$,  and the map
\[
\F_\frp:\PM\rightarrow\frp,\quad  \F_\frp (\nu)=\int_M \mu_{\mathfrak p} (x) \mathrm d \nu (x),
\]
is the analogue of the $\G$-gradient map in this setting.
It is not hard to prove that its image coincides with the convex hull of $\mu_\frp(M)$ in $\frp$ (cf.~Lemma \ref{imgradconvhull}).

Fix a probability measure $\nu\in\PM$.
Having in mind the classical convexity results for the momentum mapping \cite{Ati,GuiSte,Kir2} and for the $\G$-gradient map \cite{HeiSchu, HeiSto2},
in this paper we are interested in studying the behaviour of $\F_\frp$ on the orbit $\G\cdot \nu$.

Let $\fra \subset \frp$ be an Abelian subalgebra of $\frg$. The Abelian Lie group $\A \coloneqq \exp (\fra )$ is compatible and the corresponding $\A$-gradient map
is given by $\mu_\fra \coloneqq \pi_\fra\circ\mu_\frp$, where $\pi_\fra$ is the orthogonal projection onto $\fra$.
In Section \ref{sectabres}, we prove a result which can be regarded as the analogue of a theorem by Atiyah \cite{Ati} in our setting (see also  \cite{HeiSto2}):
\begin{theorem*}
The image of the map $\F_{\fra}:\A \cdot \nu \rightarrow \fra$ is an open convex subset of an affine subspace of $\fra$ with direction $\fra_\nu^\perp$.
Moreover, $\F_\fra (\overline{\A \cdot \nu} )$ is the convex hull of $\F_\fra (\overline{\A \cdot \nu} \cap \PM^\A )$, where $\PM^\A$ is the set of $\A$-fixed measures.
\end{theorem*}

As an immediate consequence of this theorem, we get that $\Fa(\A\cdot\nu)$ is a convex subset of $\fra$
whenever the Lie algebra of the isotropy group $\A_\nu$ is trivial (Corollary \ref{corabthm}).
The image of the map $\F_\fra$ is contained in the convex hull of $\mu_{\fra}(M)$.
Hence, when $P\coloneqq\mu_\fra (M)$ is a polytope, it is natural to investigate under which conditions $\F_{\fra} (A \cdot \nu)$ coincides with $\mathrm{int}(P)$. 
We point out that the convexity of $P$ is not known for a generic $\A$-invariant closed submanifold $M$ of $Z$.
It holds if $\G=\U^\C$ and $M$ is a complex connected submanifold by the Atiyah-Guillemin-Sternberg convexity theorem \cite{Ati,GuiSte},
or, more in general, if $Z$ is a Hodge manifold and $M$ is an irreducible semi-algebraic subset of $Z$ with irreducible real algebraic Zariski closure \cite{BilGhiHei2,HeiSchu}.
In the recent paper \cite{bgkn}, the authors gave a short proof of this property when $M$ is an $\A$-invariant compact connected real analytic submanifold of $\mathbb P^n (\C)$.
The key point is that for any $\beta \in \fra$ the Morse-Bott function $\mu_\frp^\beta$ has a unique local maximum.
Under this assumption, in Theorem \ref{AbThmP} we show that if $\A_\nu$ is trivial and for any $\beta \in \fra$ the unstable manifold corresponding to the unique maximum
of $\mu_{\frp}^\beta$ has full measure, then $\F_{\fra} (A \cdot \nu)$ coincides with $\mathrm{int}(P)$. 
It is worth underlining here that a further result shown in \cite{bgkn} allows to obtain an alternative proof of the convexity properties of the map $\Fa$ along the $\A$-orbits.  
Nevertheless, in our proof the image of $\Fa$ along the orbits is better understood. Moreover, it is completely determined for a large class of probability measures in 
Theorem \ref{AbThmP}.

In Section \ref{convgencase}, we focus our attention on the non-Abelian case.
Let $\Omega(\mu_\frp)$ denote the interior of the convex hull of $\mu_\frp(M)$ in $\frp$.
In Theorem \ref{ThmNAb}, we prove that, under a mild regularity assumption on the measure $\nu$,  $\F_\frp(\G\cdot\nu)=\Omega(\mu_\frp)$ and that
the map
\[
F_\nu:\G \rightarrow \Omega (\mu_{\frp}),\quad F_\nu(g)\coloneqq\F_\frp(g\cdot\nu),
\]
is a smooth fibration. Notice that the assumptions in Theorem \ref{AbThmP} are weaker than those of Theorem \ref{ThmNAb}.
Finally, if $\nu$ is a $\K$-invariant smooth measure on $M,$ we show that the map $F_\nu$ descends to a map on $\G/\K$ which is a diffemorphism onto $\Omega(\mu_\frp)$. 
(Corollary \ref{cordescdiffeo}). 
These results may be regarded as a generalization of those obtained in \cite{BilGhi} when $\G=\U^\C$ and $M=Z$ is a K\"ahler manifold. 
However, our proofs are slightly different, since the real case is more involved than the complex one and a new technical result is needed (cf.~Appendix \ref{apdx}).  
Moreover, Corollary \ref{cordescdiffeo} suggests that when $M$ is an adjoint orbit and $\nu$ is a $\K$-invariant smooth measure, then a potential compactification of $\G/\K$ 
is given by the convex hull of $M.$ This is an analogue of a classical result due to Kor\'anyi \cite{Kor}.

The present paper is organized as follows. In Section \ref{prelimsect}, we review the main properties of compatible groups and of the $\G$-gradient map.
In Section \ref{sectmeas}, we recall some useful results on measures and we introduce the gradient map.
The convexity properties of the gradient map in the Abelian and in the non-Abelian case are investigated in Section \ref{sectabres} and in Section \ref{convgencase}, respectively.
Finally, in Appendix \ref{apdx}, we prove a technical result which is of interest in Section \ref{convgencase}.

\section{Preliminaries}\label{prelimsect}
\subsection{Cartan decomposition and compatible subgroups}\label{Cartdec}
Let $\U$ be a compact connected Lie group, denote by $\fru$ its Lie algebra and by $\U^\C$ its complexification.
It is well-known (see for instance \cite{Kna}) that $\U^\C$ is a complex reductive Lie group with Lie algebra $\fru^\C=\fru\oplus i\fru$
and that it is diffeomorphic to $\U\times i\fru$ via the real analytic map
\[
\U\times i\fru\rightarrow\U^\C,\quad (u,i\xi)\mapsto u\exp(i\xi).
\]
The resulting decomposition $\U^\C=\U\exp(i\fru)$ is called {\em Cartan decomposition} of $\U^\C$.

A closed connected subgroup $\G\subseteq\U^\C$ with Lie algebra $\frg$ is said to be {\em compatible} with the Cartan decomposition of $\U^\C$
if $\G=\K\exp(\frp)$, where $\K\coloneqq \G\cap\U$ and $\frp \coloneqq \frg\cap i\fru$ is a $\K$-stable linear subspace of $i\fru$ (cf.~\cite{HeiSchSto,HeiSto}).
In such a case,  $\K$ is a maximal compact subgroup of $\G$.
The Lie algebra of $\G$ splits as $\frg=\frk\oplus\frp$, where $\frk \coloneqq \mbox{Lie(K)}$, and the following inclusions hold
\[
[\frk,\frk]\subset\frk,\quad[\frk,\frp]\subset\frp,\quad [\frp,\frp]\subset\frk.
\]

On the Lie algebra $\fru^\C=\fru\oplus i\fru$ there exists a nondegenerate, $\Ad(\U^\C)$-invariant, symmetric $\R$-bilinear form
$B:\fru^\C\times\fru^\C\rightarrow \R$ which is positive definite on $i\fru$, negative definite on $\fru$ and such that the decomposition $\fru\oplus i\fru$ is $B$-orthogonal
(see~e.g.~\cite[p.~585]{BilGhiHei}). In what follows, we let $\langle\cdot,\cdot\rangle \coloneqq B|_{i\fru\times i\fru}$.

Whenever $\G=\K\exp(\frp)$ is a compatible subgroup of $\U^\C$, the restriction of $B$ to $\frg$ is $\Ad(\K)$-invariant, positive definite on $\frp$, negative definite on $\frk$,
and fulfils $B(\frk,\frp)=0$.

\subsection{The G-gradient map}\label{grmapsect}
Let $\U$ and $\U^\C$ be as in $\S$\ref{Cartdec}.
Consider a compact K\"ahler manifold $(Z,J,\omega)$, assume that $\U^\C$ acts holomorphically on it and that a Hamiltonian action of $\U$ on $Z$ is defined.
Then, the K\"ahler form $\omega$ is $\U$-invariant and there exists a {\it momentum mapping} $\mu:Z\rightarrow\fru^*$.
By definition, $\mu$ is $\U$-equivariant and for each $\xi\in\fru$
\[
d\mu^\xi=\iota_{\xi_Z}\omega,
\]
where $\mu^\xi\in\mathcal{C}^\infty(Z)$ is defined by $\mu^\xi(z)=\mu(z)(\xi)$, for every point $z\in Z$, and
$\xi_Z\in\mathfrak{X}(Z)$ is the {\it fundamental vector field} of $\xi$ induced by the $\U$-action, namely
the vector field on $Z$ whose value at $z\in Z$ is
\[
\xi_Z(z)=\left.\frac{d}{dt}\right|_{t=0}\exp(t\xi) \cdot z.
\]

Since $\U$ is compact, we can identify $\fru^*$ with $\fru$ by means of an $\Ad(\U)$-invariant scalar product on $\fru$. Consequently, we can regard $\mu$ as a $\fru$-valued map.

Let $\G=\K\exp(\frp)$ be a compatible subgroup of $\U^\C$. The composition of $\mu$ with the orthogonal projection of $\fru$ onto $i\frp\subset\fru$ defines a
$\K$-equivariant map $\mu_{i\frp}:Z\rightarrow i\frp,$ which represents the analogue of $\mu$ for the $\G$-action.
Following \cite{HeiSchu,HeiSchSto,HeiSto}, in place of $\mu_{i\frp}$ we consider
\[
\mu_\frp:Z\rightarrow \frp,\quad \mu_\frp(z) \coloneqq -i\,\mu_{i\frp}(z).
\]
As the $\U^\C$-action on $Z$ is holomorphic, for every $\beta\in\frp$ the fundamental vector field $\beta_Z\in\mathfrak{X}(Z)$
induced by the $\G$-action is the gradient of the function
\[
\mu^\beta_\frp:Z\rightarrow\R,\quad \mu^\beta_\frp(z) \coloneqq  \langle \mu_\frp(z),\beta\rangle,
\]
with respect to the Riemannian metric $\omega(\cdot,J\cdot)$. This motivates the following.
\begin{definition}
$\mu_\frp$ is called {\em $\G$-gradient map} associated with $\mu$.
\end{definition}

Let $M$ be a $\G$-stable submanifold of $Z.$ We use the symbol $\mu_\frp$ to denote the $\G$-gradient map restricted to $M,$ too.
Then, for any $\beta\in\frp$ the fundamental vector field $\beta_M\in\mathfrak{X}(M)$
is the gradient of $\mu^\beta_\frp:M\rightarrow\R$ with respect to the induced Riemannian metric on $M.$
Moreover, if $M$ is compact, $\mu^\beta_\frp$ is a Morse-Bott function (see e.g.~\cite[Cor.~2.3]{BilGhiHei}). Thus, denoted by $c_1<\cdots<c_r$ the critical values of $\mu^\beta_\frp$,
$M$ decomposes as
\begin{equation}\label{MBdec}
M=\bigsqcup_{j=1}^rW_j,
\end{equation}
where for each $j=1,\ldots,r,$ $W_j$ is the unstable manifold of the critical component $(\mu_\frp^\beta)^{-1}(c_j)$ for the gradient flow of $\mu^\beta_\frp$
(see for instance \cite{HeiSchw,HeiSchSto} for more details).

\section{Measures}\label{sectmeas}
In the first part of this section we recall some known results about measures. The reader may refer for instance to \cite{DunSch,Fol} for more details.

Let $M$ be a compact manifold and let $\MM$ denote the vector space of finite signed Borel measures on $M.$
By \cite[Thm.~7.8]{Fol}, such measures are Radon.
Then, by the Riesz Representation Theorem \cite[Thm.~7.17]{Fol},
$\MM$ is the topological dual of the Banach space $(\CM,\left\|\cdot\right\|_\infty)$,
namely the space of real valued continuous functions on $M$ endowed with the sup-norm.
As a consequence, $\MM$ is endowed with the weak$^*$ topology \cite[p.~169]{Fol}.

The set of {\em Borel probability measures} on $M$ is the compact convex subset $\PM\subset\MM$ given by the intersection of the cone of positive measures on $M$
and the affine hyperplane $\{\nu\in\MM\st\nu(M)=1\}$.
Observe that the weak$^*$ topology on $\PM$ is metrizable, since $\CM$ is separable \cite[p.~426]{DunSch}.

Given a measurable map $f:M\rightarrow N$ between measurable spaces and a measure $\nu$ on $M,$ the {\em image measure} $f_*\nu$ of $\nu$ is the measure on $N$ defined by
$f_*\nu(A) \coloneqq \nu(f^{-1}(A))$ for every measurable set $A\subseteq N.$ $f_*\nu$ satisfies the following {\em change of variables formula}
\begin{equation}\label{changevar}
\int_Nh(y)d(f_*\nu)(y)=\int_M h(f(x))d\nu(x).
\end{equation}

When a Lie group $\G$ acts continuously  on a compact manifold $M,$ it is possible to define an action of $\G$ on $\PM$ as follows:
\begin{equation}\label{actGPM}
\G\times\PM\rightarrow\PM,\quad (g,\nu)\mapsto g_*\nu \coloneqq (\mathcal{A}_g)_*\nu,
\end{equation}
where for each $g\in\G$
\[
\mathcal{A}_g:M\rightarrow M,\quad\mathcal{A}_g(x)=g\cdot x,
\]
is the homeomorphism induced by the $\G$-action on $M.$
By \cite[Lemma 5.5]{BilGhi}, the action \eqref{actGPM} is continuous with respect to the weak$^*$ topology on $\PM.$
In what follows, we denote this action by a dot, i.e., $g\cdot\nu \coloneqq g_*\nu$ whenever $g\in\G$ and $\nu\in\PM.$

The next lemma is an immediate consequence of \cite[Lemma 5.8]{BilGhi}.
\begin{lemma}\label{lemmanae}
Let $M$ be a compact manifold endowed with a smooth action of a Lie group $\G$. Consider $\nu\in\MM$, $\xi\in\frg$, and suppose that
$\xi_M$ vanishes $\nu$-almost everywhere. Then, $\exp(\R\xi)$ is contained in the isotropy group $\G_\nu$ of $\nu$.
\end{lemma}
\begin{proof}
Since $\xi_M$ vanishes $\nu$-almost everywhere, its flow
\[
\f_t:M\rightarrow M,\quad \f_t(x)=\exp(t\xi)\cdot x,
\]
satisfies ${\f_t}_*\nu=\nu$ for any $t\in\R$ by \cite[Lemma 5.8]{BilGhi}.
\end{proof}
Let us focus on the setting $(M,\G,\K,\mu_\frp)$ introduced at the end of $\S$\ref{grmapsect}.
From now on, we assume that the $\G$-stable submanifold $M\subset Z$ is compact.
By the above results, the group $\G=\K\exp(\frp)$ acts continuously on $\PM$.
Moreover, albeit a reasonable symplectic structure on $\PM$ does not seem to exist,
it is possible to define a map which can be regarded as the analogue of the $\G$-gradient map $\mu_\frp$ for the action of $\G$ on $\PM$.
\begin{definition}
The {\em gradient map} associated with the action of $\G$ on $\PM$ is
\[
\F:\PM\rightarrow\frp,\quad \F(\nu)=\int_M\mu_\frp(x)d\nu(x).
\]
\end{definition}

\begin{remark}
By \cite[Prop.~45]{BilZed}, $\F$ is precisely the gradient map of a Kempf-Ness function for $(\PM,\G,\K)$.
Thus, it is continuous and $\K$-equivariant (cf.~\cite[Sect.~3]{BilZed}).
\end{remark}

Using $\F$, the usual concepts of stability \cite{GeoRobSal,Hei,HeiHuc,HeiHucLoo,HeiLoo,KemNes,Kir,MumFogKir,Schw,Sja} can be defined for probability measures, too
(see also \cite{BilGhi,BilZed}).
For instance, a measure $\nu\in\PM$ is said to be {\em stable} if
\[
\G \cdot \nu\cap\F^{-1}(0)\neq\emptyset
\]
and $\frg_\nu \coloneqq \mbox{Lie}(\G_\nu)$ is conjugate to a subalgebra of $\frk$. In such a case, $\G_\nu$ is compact \cite[Cor.~3.5]{BilGhi}.

In the light of previous considerations, it is natural to ask whether established results for the
$\G$-gradient map \cite{Ati,Dui,GuiSte,HeiSchu,HeiSto2} can be proved also for the gradient map $\F$.
Here, we focus our attention on convexity properties of $\F$. We begin with the following observation.
\begin{lemma}\label{imgradconvhull}
The image of the gradient map $\F:\PM\rightarrow\frp$ coincides with the convex hull $E(\mu_\frp)$ of $\mu_\frp(M)$ in $\frp$.
\end{lemma}
\begin{proof}
Consider $\nu\in\PM$. Observe that $\F(\nu)$ is the barycenter of the measure ${\mu_\frp}_*\nu\in\PR(\mu_\frp(M))$,
since by the change of variables formula \eqref{changevar} we have
\[
\F(\nu)=\int_M\mu_\frp(x)d\nu(x) = \int_{\frp} \beta\,d({\mu_\frp}_*\nu)(\beta).
\]
Thus, $\F(\nu)$ lies in $E(\mu_\frp)$.
Conversely, for any $\gamma\in E(\mu_\frp)$, we can write
\[
\gamma=\sum_{j=1}^m\lambda_j\gamma_j,
\]
for a suitable $m$, where $\sum_{j=1}^m\lambda_j=1$, $\lambda_j\geq0$ and $\gamma_j\in\mu_\frp(M)$.
For each $j=1,\ldots,m$, let $x_j\in M$ be a point in the preimage of $\gamma_j$
and let $\delta_{x_j}$ denote the Dirac measure supported at $x_j$. Then, $\gamma=\F(\widetilde{\nu})$, where
\[
\widetilde{\nu} \coloneqq \sum_{j=1}^m\lambda_j\delta_{x_j}.
\]
\end{proof}

Due to the previous result, in the next sections we shall study the behaviour of $\F$ on the orbits of the $\G$-action.

\section{Convexity properties of $\F$: Abelian case}\label{sectabres}
Let $\fra\subset\frp$ be a Lie subalgebra of $\frg$.
Since $[\frp,\frp]\subset\frk$ and $\frg=\frk\oplus\frp$, $\fra$ is Abelian.
The corresponding Abelian Lie group $\A \coloneqq \exp(\fra)\subset\G$ is compatible with the Cartan decomposition of $\U^\C$
and an $\A$-gradient map $\mu_\fra:M\rightarrow\fra$ is given by $\mu_\fra \coloneqq \pi_\fra \circ \mu_\frp$, where $\pi_\fra$ is the orthogonal projection onto $\fra$.
Therefore, the gradient map associated with the $\A$-action on $\PM$ is
\[
\Fa:\PM\rightarrow\fra,\quad \Fa(\nu)=\int_M\mu_\fra(x)d\nu(x).
\]

Fix a probability measure $\nu\in\PM$. We want to study the behaviour of $\Fa$ on the orbit $\A\cdot\nu$.
First of all, we show that $\A_\nu$ is always compatible.
\begin{lemma}
The isotropy group $\A_\nu$ of $\nu$ is compatible, namely $\A_\nu=\exp(\fra_\nu)$.
\end{lemma}
\begin{proof}
Let ${\alpha} \coloneqq \Fa(\nu)\in\fra$.
Since $\fra$ is Abelian, $\widetilde{\mu}_\fra \coloneqq \mu_\fra-{\alpha}$ is still an $\A$-gradient map and the corresponding gradient map
$\widetilde{\F}_\fra:\PM\rightarrow \fra$ satisfies
\[
\widetilde{\F}_\fra(\nu)=\int_M\widetilde{\mu}_\fra(x)d\nu(x)=\Fa(\nu)-{\alpha}\,\nu(M)=0.
\]
Then, $\A_\nu$ is compatible by \cite[Prop.~20]{BilZed}.
\end{proof}

Consider the decomposition
\[
\fra=\fra_\nu\oplus\franp,
\]
where $\fra_\nu^\perp$ is the orthogonal complement of $\fra_\nu$ in $\fra$ with respect to $B|_{\fra\times\fra}$.
We denote by $\pi:\fra\rightarrow\franp$ the orthogonal projection onto $\franp$ and we let $\hat\A \coloneqq \exp(\franp)$.
Since $\exp:\fra\rightarrow\A$ is an isomorphism of Abelian Lie groups, we have $\A=\hat{\A}\A_\nu$ and $\A\cdot\nu=\hat{\A}\cdot\nu.$

We are now ready to state the main result of this section.
\begin{theorem}\label{ThmAb}
The image $\Fa(\A\cdot\nu)$ of the orbit $\A\cdot\nu$ is an open convex subset of an affine subspace of $\fra$ with direction $\franp$.
\end{theorem}

Before proving Theorem \ref{ThmAb}, we show a preliminary lemma.
\begin{lemma}\label{convexproj}
The projection of $\Fa(\hat{\A}\cdot \nu)$ onto $\franp$ is convex.
\end{lemma}
\begin{proof}
By \cite[Thm.~39]{BilZed}, there exists a Kempf-Ness function $\Psi:M\times\A\rightarrow\R$ for $(M,\A,\{e\})$, where $e\in\A$ is the identity element.
Recall that for each point $x\in M$ the function $\Psi(x,\cdot)$ is smooth on $\A$, and that for every $\gamma\in\fra$
\begin{equation}\label{KNconv}
\frac{d^2}{dt^2}\Psi(x,\exp(t\gamma))\geq0,
\end{equation}
and it vanishes identically if and only if $\exp(\R\gamma)\subset \A_x$. Moreover, for every $a,b\in\A,$ the following condition is satisfied
\begin{equation}\label{KNcocy}
\Psi(x,ab) = \Psi(x,b)+\Psi(b\cdot x,a).
\end{equation}
$\Psi$ is related to the $\A$-gradient map $\mu_\fra$ by
\begin{equation}\label{KNdefn}
\dtz\Psi(x,\exp(t\gamma))=\langle\mu_\fra(x),\gamma \rangle.
\end{equation}
We define a function $f:\franp\rightarrow\R$ as follows
\[
f(\alpha) \coloneqq \int_M\Psi(x,\exp(\alpha))d\nu(x).
\]
We claim that $f$ is strictly convex. By \eqref{KNconv} and \eqref{KNcocy}, for every $\alpha,\beta\in\franp$
\[
\frac{d^2}{dt^2}f(t\beta+\alpha) = \int_M\frac{d^2}{dt^2}\Psi(\exp(\alpha)\cdot x,\exp(t\beta))d\nu(x) \geq0.
\]
If it was identically zero, then $\frac{d^2}{dt^2}\Psi(\exp(\alpha)\cdot x,\exp(t\beta))$ would vanish $\nu$-almost everywhere.
As a consequence, for every point $x$ outside a set of $\nu$-measure zero we would have
$\exp(\R\beta)\subset\A_{\exp(\alpha)\cdot x} = \A_x$, which implies that $\beta_M(x)=0$.
Therefore, $\exp(\R\beta)\subset\A_\nu$ by Lemma \ref{lemmanae}, which is a contradiction.
By a standard result in convex analysis (see for instance \cite[p.~122]{Gui}), the pushforward $df:\franp\rightarrow(\franp)^*$ is a diffeomorphism onto an open convex subset of $(\franp)^*$.
Now, using \eqref{changevar}, \eqref{KNcocy}, \eqref{KNdefn}, for each $\alpha,\beta\in\franp$ we have
\begin{eqnarray*}
df(\alpha)(\beta) 	&=&\dtz f(t\beta+\alpha)\\
				&=& \int_M\dtz\Psi(\exp(\alpha)\cdot x,\exp(t\beta))d\nu(x)\\
				&=& \langle\int_M\mu_\fra(\exp(\alpha)\cdot x)d\nu(x),\beta\rangle\\
				&=& \langle\int_M\mu_\fra(y)d(\exp(\alpha)\cdot\nu)(y),\beta\rangle\\
				&=& \langle \Fa(\exp(\alpha)\cdot\nu),\beta\rangle\\
				&=& \langle \pi(\Fa(\exp(\alpha)\cdot\nu)),\beta\rangle,
\end{eqnarray*}
from which the assertion follows.
\end{proof}

\begin{corollary}\label{corabthm}
If $\fra_\nu=\{0\},$ then $\Fa(\A\cdot\nu)$ is convex in $\fra$ and the map
\[
F^{\A}_\nu:\A\rightarrow\fra,\quad F^{\A}_\nu(a) \coloneqq \Fa(a\cdot\nu),
\]
is a diffeomorphism onto $\Fa(\A\cdot\nu)$.
\end{corollary}

\begin{proof}[Proof of Theorem \ref{ThmAb}]
Since $\A_\nu$ is compatible, it follows from the proof of \cite[Prop.~52]{BilZed} that $\nu$ is supported on
\[
M^{\fra_\nu} \coloneqq \{x\in M\st \xi_M(x)=0~\forall~\xi\in\fra_\nu\}.
\]
By \cite{HeiSchw,HeiSchSto}, there exists a decomposition
\[
M^{\fra_\nu} = M_1\sqcup\cdots\sqcup M_n,
\]
where each $M_j$ is an $\A$-stable connected submanifold of $M.$ Consequently,
\[
\nu=\sum_{j=1}^n\lambda_j\nu_j,
\]
where for $j=1,\ldots,n$, $\nu_j$ is a probability measure on $M_j$,  $\lambda_j\geq0$ and $\sum_{j=1}^n\lambda_j=1$.
By \cite{HeiSto}, for every $x\in M_j$ the image $\mu_\fra(\A\cdot x)$ of $\A\cdot x$ is contained in an affine subspace $\alpha_j+\franp$ of $\fra$.
Then, since $M_j$ is $\A$-stable, there is a map $\widetilde{\mu}_j:M_j\rightarrow\franp$ such that $\mu_\fra(a\cdot x) = \alpha_j+\widetilde{\mu}_j(a\cdot x)$, for every $a\in\A$.
Now, we have
\begin{eqnarray*}
\Fa(a\cdot\nu)   &=& \int_M\mu_\fra(x)d(a\cdot\nu)(x)\\
			&=& \int_M\mu_\fra(a\cdot x)d\nu(x)\\
			&=&\sum_{j=1}^n\lambda_j\int_{M_j}\mu_\fra(a\cdot x)d\nu_j(x)\\
			&=&\sum_{j=1}^n\lambda_j\alpha_j + \sum_{j=1}^n\lambda_j\int_{M_j}\tilde{\mu}_j(a\cdot x) d\nu_j(x).
\end{eqnarray*}
Hence, $\Fa(\A\cdot\nu)\subseteq\alpha+\franp$, where $\alpha \coloneqq \sum_{j=1}^n\lambda_j\alpha_j$.
Using Lemma \ref{convexproj}, we can conclude that $\Fa(\A\cdot\nu)$ is an open convex subset of the affine subspace $\alpha+\franp$ of $\fra$.
\end{proof}

From the previous result and the compactness of $\PM$, it follows that $\Fa(\overline{\A\cdot\nu})=\overline{\Fa(\A\cdot\nu)}$ is a compact convex subset of $\fra$.
Moreover, if we denote by
\[
\PM^\A \coloneqq \{\nu\in\PM\st \A\cdot\nu=\nu\}
\]
the set of $\A$-fixed measures, then we have the
\begin{proposition}
$\Fa\left(\overline{\A\cdot\nu}\right)$ is the convex envelope of $\Fa(\overline{\A\cdot\nu}\cap\PM^\A)$.
\end{proposition}
\begin{proof}
By \cite[Cor.~1.4.5]{Sch}, it is sufficient to show that every extremal point $\beta\in\Fa\left(\overline{\A\cdot\nu}\right)$ is the image of an $\A$-fixed measure.
Consider $\widetilde\nu\in\overline{\A\cdot\nu}$ such that $\Fa(\widetilde\nu)=\beta$.
By Theorem \ref{ThmAb}, $\Fa(\A\cdot\widetilde\nu)$ is an open convex subset of an affine subspace $\alpha+\fra_{\widetilde\nu}^\perp\subset\fra$.
Since $\beta$ is an extremal point, we have necessarily $\fra_{\widetilde\nu}^\perp=\{0\}$. Thus, $\widetilde\nu\in\PM^\A$.
\end{proof}

Let  $P \coloneqq \mu_\fra(M)$. It was proved in \cite[Sect.~5]{HeiSchu} that $P$ is a finite union of polytopes, while in \cite{BilGhiHei3}
the authors showed that its convex hull is closely related to $E(\mu_\frp)$.
Moreover, even if $P$ is not necessarily convex, there exist suitable hypothesis guaranteeing that it is a polytope.
This happens for instance if for each $\beta\in\fra$ any local maximum of the Morse-Bott function $\mu_\frp^\beta$ is a global maximum \cite{bgkn}.
Classes of manifolds satisfying this property include real flag manifolds \cite{BilGhiHei}, and real analytic submanifolds of the complex projective space \cite{bgkn}.

In the sequel, we always assume that for each $\beta \in \mathfrak a$ the function $\mu_\frp^\beta$ has a unique local maximum.
As a consequence, $P$ is a polytope, and the Morse-Bott decomposition \eqref{MBdec} of $M$ with respect to $\mu_\frp^\beta$ has a unique unstable manifold which is open and dense,
namely  $W_r$, while the remaining unstable manifolds are submanifolds of positive codimension.

\begin{definition}
Let $\mathscr{W}(M,\A)$ denote the set of probability measures on $M$ for which the open unstable manifold $W_r$ has full measure for every $\beta\in\fra$.
\end{definition}

A typical example of probability measures belonging to $\mathscr{W}(M,\A)$ is given by {\em smooth} ones, namely those having
a smooth positive density in any chart of the manifold with respect to the Lebesgue measure of the chart  (cf.~\cite[Sect.~11.4]{Fol}).

In a similar way as in \cite[Prop.~6.8]{BilGhi}, we can prove the following
\begin{theorem}\label{AbThmP}
Let $\nu\in\mathscr{W}(M,\A)$ and assume that $\A_\nu=\{e\}$. Then, $\Fa(\A\cdot\nu)$ coincides with ${\rm int}(P)$.
\end{theorem}
\begin{proof}
For simplicity of notation, let $\Os \coloneqq \Fa(\A\cdot\nu)\subset\fra$.
We already know that $\Os\subseteq \rm{int}(P)$.
Suppose by contradiction that  $\Os$ is strictly contained in $\rm{int}(P)$.
Then, $\overline{\Os}\subset P$, since $\Os$ and $P$ are both convex. Consider $\alpha_0\in P- \overline{\Os}$, $\alpha_1\in\Os$ and the line segment
$\sigma(t) \coloneqq (1-t)\,\alpha_0+t\,\alpha_1$. Let $\ts \coloneqq {\rm inf}\{t\in[0,1]\st \sigma(t)\in\overline{\Os}\}$ and $\overline\alpha \coloneqq \sigma(\ts)$.
As $\overline{\Os}$ is closed, $\overline\alpha\in\overline{\Os}$ and $\ts\in(0,1)$.
We claim that $\overline\alpha\in\partial\Os\cap{\rm int}(P)$. Indeed, it is clear that $\overline\alpha\in\partial\Os$,
while $\overline\alpha\in {\rm int}(P)$ follows from $\alpha_1\in\Os\subset{\rm int}(P)$ and $\ts>0$.
By \cite{Sch}, every boundary point of a compact convex set lies on an exposed face, that is, it admits a support hyperplane.
Therefore, there exists $\beta\in\fra$ such that
\[
\langle\overline\alpha,\beta \rangle = \max_{\alpha\in\overline{\Os}}\langle\alpha,\beta\rangle = \sup_{\alpha\in\Os}\langle\alpha,\beta\rangle =
\sup_{\gamma\in\fra}\langle \Fa(\exp(\gamma)\cdot\nu),\beta\rangle.
\]
Since $\nu\in\mathscr{W}(M,\A)$ and $\mu_\frp^\beta=\mu_\fra^\beta$ for every $\beta\in\fra$,
it follows from \cite[Cor.~54]{BilZed} and from the proof of \cite[Thm.~53]{BilZed} that
\[
\max_{M}\mu_\fra^\beta = \lim_{t\rightarrow+\infty}\int_M\mu_\fra^\beta(\exp(t\beta)\cdot x)d\nu(x) = \lim_{t\rightarrow+\infty}\langle\Fa(\exp(t\beta)\cdot\nu),\beta\rangle.
\]
Consequently,
\[
\langle\overline\alpha,\beta \rangle = \sup_{\gamma\in\fra}\langle \Fa(\exp(\gamma)\cdot\nu),\beta\rangle \geq \max_{M}\mu_\fra^\beta = \max_{\rho\in P}\langle \rho,\beta\rangle.
\]
That being so, the linear function $\alpha\mapsto\langle\alpha,\beta\rangle$ attains it maximum on $P$ at $\overline\alpha\in{\rm int}(P)$.
Since $P$ is convex, $\beta$ must be zero, which is a contradiction.
\end{proof}

\section{Convexity properties of $\F$: general case}\label{convgencase}
The goal of this section is to prove a result similar to Theorem \ref{AbThmP} when the group acting on $\PM$ is non-Abelian.

Let $\G=\K\exp(\frp)$ be a compatible subgroup of $\U^\C$ and fix $\nu\in\PM$.
To our purpose, it is useful to consider the map \cite{BilGhi0, BilGhi, BouLiYau, Her}
\[
F_\nu:\G\rightarrow\frp,\quad F_\nu(g)\coloneqq\F(g\cdot\nu),
\]
where $\F:\PM\rightarrow\frp$ is the gradient map associated with the action of $\G$ on $\PM$.
In \cite[Thm.~6.4]{BilGhi}, the authors showed that $F_\nu$ is a smooth submersion when $\G=\U^\C$ and $\G_\nu$ is compact.
This is true for a compatible subgroup of $\U^\C$, too.
\begin{proposition}\label{Fnufibr}
If $\G_\nu$ is compact, then $F_\nu$ is a smooth submersion.
\end{proposition}
\begin{proof}
We have to prove that the pushforward $dF_\nu(g):T_g\G\rightarrow\frp$ of $F_\nu$ is surjective for every $g\in\G$.
Let us consider the curve $\sigma(t)\coloneqq\exp(t\beta)\cdot g$ in $\G$, where $\beta\in\frp$. Using the change of variables formula \eqref{changevar}, we can write
\[
F_\nu(\sigma(t))=\int_M\mu_\frp(\exp(t\beta)\cdot x)d\widetilde\nu(x),
\]
where $\widetilde\nu \coloneqq g\cdot\nu\in\PM$. Suppose that $dF_\nu(g)(\dot\sigma(0)) =0$. Then, denoted by $\left\|\cdot\right\|$ the Riemannian norm on $M,$ we have
\[
0 = \langle dF_\nu(\dot\sigma(0)),\beta\rangle = \int_M\dtz\mu_\frp^\beta(\exp(t\beta)\cdot x)\,d\widetilde\nu(x) = \int_M\left\| \beta_M\right\|^2(x)\,d\widetilde\nu(x),
\]
since ${\rm grad}(\mu_\frp^\beta)=\beta_M.$
Therefore, $\beta_M$ vanishes $\widetilde\nu$-almost everywhere. By Lemma \ref{lemmanae}, $\exp(\R\beta)$ is contained in $\G_{\widetilde\nu}=g\,\G_\nu g^{-1},$
which is compact. Thus, $\beta=0$. We can conclude that $dF_\nu(g)$ is injective on the subspace $dR_g(e)(\frp)$ of $T_g\G$, $R_g$ being the right translation on $\G$.
By dimension reasons, $dF_\nu(g)$ is surjective.
\end{proof}

As in the previous section, whenever $\fra\subset\frp$ is a maximal Abelian subalgebra of $\frg$ with corresponding Abelian Lie group $\A \coloneqq \exp(\fra)$,
we assume that the Morse-Bott function $\mu_\frp^\beta$ has a unique local maximum for every $\beta\in\fra$.
In the non-Abelian case, we can exploit the so-called $\K\A\K$ {\em decomposition} of $\G$ (cf.~\cite[Thm.~7.39]{Kna}) to show the following.
\begin{theorem}\label{ThmNAb}
Let $\nu\in\PM$ be a probability measure which is absolutely continuous with respect to a $\K$-invariant smooth probability measure $\nu_0\in\PM$
and assume that $0$ belongs to the interior $\Omega(\mu_\frp)$ of $E(\mu_\frp)$ in $\frp$.
Then, $\F(\G\cdot\nu) = \Omega(\mu_\frp)$ and $F_\nu:\G\rightarrow\Omega(\mu_\frp)$ is a smooth fibration with compact connected fibres diffeomorphic to $\K$.
\end{theorem}

Before proving the theorem, we make some remarks on its content.
First, we observe that the hypothesis on $\nu$ is satisfied by smooth probability measures, which constitute a dense subset of $\PM$
(see for instance \cite{DunSch}).
Moreover, it guarantees that whenever $\{k_n\}$ is a sequence in $\K$ converging to some $k\in\K$, then
the sequence $\{k_n\cdot\nu\}\subset\PM$ converges to $k\cdot\nu$ in the norm
\[
\left\|\nu\right\| \coloneqq \sup\left\{\int_M hd\nu\st h\in\CM,~\sup_M |h|\leq1 \right\},
\]
by \cite[Lemma 6.11]{BilGhi}.
Finally, we underline that the assumption $0\in\Omega(\mu_\frp)$ is not restrictive, as
such condition is always satisfied up to replace $\G$ with a compatible group $\G'=\K'\exp(\frp')$ such that
$\mu_{\frp'}(M)=\mu_\frp(M)$ and up to shift $\mu_{\frp'}$.
We will show this assertion in Proposition \ref{propapdx} of Appendix \ref{apdx}, since most of its proof is rather technical.

\begin{proof}[Proof of Theorem \ref{ThmNAb}]
First of all, notice that $\nu\in\mathscr{W}(M,\A)$ for any $\fra\subset\frp$, since it is absolutely continuous with respect to the smooth probability measure $\nu_0$.
As $0\in\Omega(\mu_\frp)$, for every $\beta\in\frp$ the function $\mu_\frp^\beta$ has a strictly positive maximum.
This implies that $\nu$ is stable (cf.~\cite[Cor.~56]{BilZed}). Thus, $\G_\nu$ is compact.
Now, by Proposition \ref{Fnufibr}, $F_\nu:\G\rightarrow \frp$ is a smooth submersion.
In particular, its image is an open subset of $\frp$ contained in $E(\mu_\frp)$.
Therefore, $F_\nu(\G)\subseteq\Omega(\mu_\frp)$ and we can regard $F_\nu$ as a map $F_\nu:\G\rightarrow\Omega(\mu_\frp)$. We claim that such map is proper.
Let $\{g_n\}$ be a sequence in $\G$ such that $\{F_\nu(g_n)\}$ converges to a point of $\Omega(\mu_\frp)$. We need to show that there exists a convergent subsequence of $\{g_n\}$.
Let $\fra\subset\frp$ be a maximal Abelian subalgebra of $\frg$ and set $\A \coloneqq \exp(\fra)$. By the $\K\A\K$-decomposition of $\G$, every $g_n\in\G$ can be written as
$k_n\exp(\alpha_n)\,l_n^{-1},$ where $k_n,l_n\in\K$ and $\alpha_n\in\fra.$
Passing to subsequences, we have that $k_n\rightarrow k$ and $l_n\rightarrow l$, for some $k,l\in\K$.
Since $F_\nu$ is $\K$-equivariant, it follows that the sequence $\{F_\nu(\exp(\alpha_n)\,l_n^{-1})\}$ is convergent in $\Omega(\mu_\frp)$.
A computation similar to \cite[p.~1139]{BilGhi} gives
\[
\left|F_\nu(\exp(\alpha_n)\,l_n^{-1})-F_\nu(\exp(\alpha_n)\,l^{-1})\right|\leq\sup_M|\mu_\frp|\left\|l_n^{-1}\cdot\nu-l^{-1}\cdot\nu\right\|.
\]
Then, by the hypothesis on $\nu$, we get $F_\nu(\exp(\alpha_n)\,l_n^{-1})-F_\nu(\exp(\alpha_n)\,l^{-1})\rightarrow 0$.
Therefore, the sequence $\{F_\nu(\exp(\alpha_n)\,l^{-1})\}$ is convergent in $\Omega(\mu_\frp)$, too.
Consequently,  $\{F_\nu(l\exp(\alpha_n)\,l^{-1})\}$ converges to some point of $\Omega(\mu_\frp)$, being $F_\nu(l\exp(\alpha_n)\,l^{-1}) = \Ad(l)F_\nu(\exp(\alpha_n)\,l^{-1})$.
The points $l\exp(\alpha_n)\,l^{-1}$ belong to the Abelian group $\A' \coloneqq l\A l^{-1}$, which is compatible.
The $\A'$-gradient map is $\pi_{\fra'}\circ\mu_\frp$, where $\pi_{\fra'}:\frp\rightarrow\fra'$ is the orthogonal projection onto the Lie algebra $\fra'$ of $\A'$.
Denote by $P\coloneqq\mu_{\fra'}(M)$ the image of $\mu_{\fra'}$. $P=\pi_{\fra'}(\mu_\frp(M))$ is a polytope and $\pi_{\fra'}(\Omega(\mu_\frp))\subset\mbox{int}(P)$.
Observe that $0\in\mbox{int}(P)$. This implies that $\nu$ is stable with respect to $\A'$. Thus, $\fra'_\nu=\{0\}$ by \cite[Lemma 21]{BilZed}.
Hence, by the results of $\S$\ref{sectabres},
$\F_{\fra'}(\A'\cdot\nu) = \mbox{int}(P)$ and the map $F_\nu^{\A'}:\A'\rightarrow\fra'$, $F_\nu^{\A'}(a) = \F_{\fra'}(a\cdot\nu)$, is a diffeomorphism onto $\mbox{int}(P)$.
Since $\{l\exp(\alpha_n)\,l^{-1}\}\subset\A'$ and $\{\pi_{\fra'}(F_\nu(l\exp(\alpha_n)l^{-1}))=F_\nu^{\A'}(l\exp(\alpha_n)\,l^{-1})\}$
converges to some point of $\mbox{int}(P)$, the sequences $\{l\exp(\alpha_n)l^{-1}\}\subset\A'$
and $\{\exp(\alpha_n)\}\subset\A$ admit convergent subsequences. The claim is then proved.
As a consequence, $F_\nu:\G\rightarrow\Omega(\mu_\frp)$ is a closed map. Since it is also open, it is surjective.
In particular, it is a locally trivial fibration by Ehresmann theorem \cite{Ehr}.
As the base $\Omega(\mu_\frp)$ is contractible, $\G$ is diffeomorphic to $\Omega(\mu_\frp)\times \Fg$,
where $\Fg$ denotes the fibre. Hence, $\Fg$ is connected. Moreover, $F_\nu^{-1}(0)$ is a $\K$-orbit, since $0\in\Omega(\mu_\frp)$ and $F_\nu$ is $\K$-equivariant.
Therefore, $\Fg$ is diffeomorphic to $\K$.
\end{proof}

\begin{corollary}\label{cordescdiffeo}
If $\nu\in\PM$ is a $\K$-invariant smooth probability measure on $M$ and $0\in\Omega(\mu_\frp)$, then
$F_\nu$ descends to a diffeomorphism $\overline{F}_\nu:\G/\K\rightarrow\Omega(\mu_\frp)$.
\end{corollary}
\begin{proof}
Since $\nu$ is $\K$-invariant, for every $g\in\G$ and $k\in \K$ we have $F_\nu(gk) = \F(gk\cdot\nu) = \F(g\cdot\nu) = F_\nu(g)$. Thus,
$F_\nu$ descends to a map $\overline{F}_\nu:\G/\K\rightarrow\Omega(\mu_\frp)$.
By Theorem \ref{ThmNAb}, $\overline{F}_\nu$ is a proper map and a local diffeomorphism. Thus, it is a covering map.
As $\Omega(\mu_\frp)$ is contractible, $\overline{F}_\nu:\G/\K\rightarrow\Omega(\mu_\frp)$ is a diffeomorphism.
\end{proof}

\begin{remark}
The above corollary may be regarded as an analogue of a classical result by Kor\'anyi \cite{Kor}. 
Indeed, it suggests that when $M$ is an adjoint orbit and $\nu$ is a $\K$-invariant probability measure, then a potential compactification of $\G/\K$ 
is given by the convex hull of $M.$
\end{remark}


\appendix

\section{}\label{apdx}
Let $\U$ be a compact connected Lie group acting in a Hamiltonian fashion on a compact K\"ahler manifold $(Z,J,\omega)$ with momentum mapping $\mu:Z\rightarrow\fru$,
and assume that the action of $\U^\C$ on $Z$ is holomorphic. As mentioned in $\S$\ref{convgencase}, we are going to show the following result.
\begin{proposition}\label{propapdx}
Let $\G=\K\exp(\frp)$ be a compatible subgroup of $\U^\C$. Consider a $\G$-stable submanifold $M$ of $Z$ and let $\mu_\frp:M\rightarrow\frp$
be the $\G$-gradient map associated with $\mu$. Then
\begin{enumerate}[{\rm i)}]
\item\label{propapdxi} there exists a subgroup $\G'=\K'\exp(\frp')\subset\G$ compatible with $\U^\C$ such that the interior of $\mu_{\frp'}(M)$ is nonempty in $\frp'$ and
$\mu_{\frp'}(M)=\mu_\frp(M)$;
\item\label{propapdxii}  up to shift $\mu_{\frp'}$, $0\in \Omega(\mu_{\frp'})$.
\end{enumerate}
\end{proposition}
For the sake of clarity, we first prove some lemmata which will be useful in the proof of the above proposition.

Let $\fru' \coloneqq \frk\oplus i\frp$. It is immediate to check that $\fru'$ is a subalgebra of $\fru$.
\begin{lemma}\label{LA1}
Let $\beta_0\in\frp$ be a $\K$-fixed point. Then, $[\beta_0,\frg]=0$ and $[i\beta_0,\fru']=0$.
\end{lemma}
\begin{proof}
First, observe that $[\beta_0,\frk]=0,$ since $\beta_0$ a $\K$-fixed point.
Moreover, $[\beta_0,\frp]\subset\frk$ and $B(\frk,[\beta_0,\frp])= -B([\beta_0,\frk],\frp)=0$, as $B$ is $\Ad(\U^\C)$-invariant.
Thus, $[\beta_0,\frp]=0$ and, consequently, $[\beta_0,\frg]=0$.
Finally, from the definition of $\fru'$, it follows that $[i\beta_0,\fru']=0$.
\end{proof}

Consider $\U'' \coloneqq \overline{\exp(\fru')}\subseteq\U$. $\U''$ is a compact subgroup of $\U$ and $\G$ is a compatible subgroup of $(\U'')^\C$, too.
Denote by $\fru''$ the Lie algebra of $\U''$. The momentum mapping for the $\U''$-action on $(Z,J,\omega)$ is given by $\pi_{\fru''}\circ\mu$, where $\pi_{\fru''}:\fru\rightarrow\fru''$
is the projection. Moreover, a result similar to Lemma \ref{LA1} also holds for $\fru''$.
\begin{lemma}
Let $\beta_0\in\frp$ be a $\K$-fixed point. Then, $[i\beta_0,\fru'']=0$.
\end{lemma}
\begin{proof}
Let $s\in\U''$ and let $\{\xi_n\}$ be a sequence in $\fru'$ such that $\exp(\xi_n)\rightarrow s.$ By Lemma \ref{LA1}, we have that
$\exp(it\beta_0)\exp(\xi_n)=\exp(\xi_n)\exp(it\beta_0)$, for every $t\in\R$. Therefore, $\exp(it\beta_0)\,s=s\,\exp(it\beta_0)$, that is, $\exp(it\beta_0)\in\mathcal{Z}(\U'')$.
\end{proof}

In the light of the previous observations, up to replace $\U$ with $\U''$, we can assume that $\G=\K\exp(\frp)$ is a compatible subgroup of $\U^\C$ with Lie algebra
$\frg=\frk\oplus\frp$, and that for every $\K$-fixed point $\beta_0\in\frp$ we have $[\beta_0,\frg]=0$ and $[i\beta_0,\fru]=0$.

Let us focus on the convex hull $E \coloneqq E(\mu_\frp)$ of $\mu_\frp(M)$ in $\frp$. $E$ is a $\K$-invariant convex body.
Let $\Aff(E)$ denote the affine hull of $E$. Then, $\Aff(E)=\beta_0+\frp'$, where $\frp'\subseteq\frp$ is a linear subspace.
Pick $\beta_0\in E$ such that $\|\beta_0\|=\min_{E}\|\beta\|$.
Observe that such $\beta_0$ is fixed by the $\K$-action. Therefore, $\frp'$ is $\K$-invariant.
Hence, up to shift $\mu$ by $-i\beta_0$, we may assume that $E\subseteq\frp'$ and that the interior of $E$ in $\frp'$ is nonempty.
Summarizing, we have proved the following
\begin{lemma}\label{LA3}
Up to shift the momentum mapping $\mu$, there exists a $\K$-invariant subspace $\frp'\subseteq\frp$ such that $E(\mu_\frp)$ is contained in $\frp'$ and its interior in $\frp'$ is nonempty.
\end{lemma}

\begin{proof}[Proof of Proposition \ref{propapdx}]  \
\begin{enumerate}[{\rm i)}]
\item Consider the subspace $\frp'$ of $\frp$ obtained in Lemma \ref{LA3}.
Since $\frp'$ is $\K$-invariant, $[\frp',\frp']$ is an ideal of $\frk$.
Let $\frh \coloneqq [\frp',\frp']\oplus\frp'$. The Lie algebra $\frg$ decomposes as $\frg=\frh\oplus\frh^\perp$,
where $\frh^\perp$ is the orthogonal complement of $\frh$ in $\frg$ with respect to $B$.  By \cite[Prop.~1.3]{BilGhiHei3}, $\frh$ and $\frh^\perp$ are compatible $\K$-invariant
commuting ideals of $\frg$. Set $\K_1 \coloneqq \exp([\frp',\frp'])$ and $\Hg=\K_1\exp(\frp')$. Then, the group $\G' \coloneqq \overline\Hg = \overline\K_1\exp(\frp')$
is a compatible subgroup of $\U^\C$
and the $\G'$-gradient map $\mu_{\frp'}:M\rightarrow\frp'$ associated with $\mu$ satisfies $\mu_{\frp'}(M)=\mu_\frp(M)$.
\item Let $\nu$ be a $\K$-invariant measure on $\frp'$ such that $\nu(E(\mu_{\frp'}))=1.$ Define $\theta\coloneqq\int_{E(\mu_{\frp'})}\beta d\nu(\beta)$.
$\theta$ is a $\K$-fixed point of $E(\mu_{\frp'})$. In particular, $[i\theta,\fru]=0$.
We claim that $\theta\in\Omega(\mu_{\frp'})$. Indeed, otherwise there would exist $\xi\in\frp'$ such that
$\langle\theta,\xi\rangle = c$, while $\langle\beta,\xi\rangle < c$ for every $\beta\in\Omega(\mu_{\frp'})$. From this follows that
\[
\langle\theta,\xi\rangle = \int_{E(\mu_{\frp'})}\langle \beta,\xi\rangle d\nu(\beta) =  \int_{\Omega(\mu_{\frp'})}\langle \beta,\xi\rangle d\nu(\beta)<c,
\]
which is a contradiction. Therefore, up to shift $\mu$ by $-i\theta$, we have that $0\in\Omega(\mu_{\frp'})$.
\end{enumerate}
\end{proof}

\bigskip

\noindent  {\bf Acknowledgments.} This work was done when A.R.~was a postdoctoral fellow at the Department of Mathematics and Computer Science in Parma.
He would like to thank the department for hospitality.



\end{document}